\documentclass{amsart}

\usepackage{graphicx}
\usepackage{amssymb}
\usepackage{amsrefs}
\usepackage{hyperref}

\theoremstyle{plain} %
\newtheorem{theorem}{Theorem}[section]
\newtheorem{proposition}[theorem]{Proposition}
\newtheorem{lemma}[theorem]{Lemma}

\newtheorem{remark}[theorem]{Remark}

\begin{document}

\title{Existence of stable periodic orbits in billiards close to
lemon and moon billiards}

\author{Alexander Grigo}
\email{alexander.grigo@ou.edu}
\address{Department of Mathematics, University of Oklahoma, Norman OK, USA}

\begin{abstract}
  It is known that at lemon and moon billiards that have a sufficiently
  small curvature on one of their circular arcs are hyperbolic.
  In this paper we show that replacing this circular arc by a more
  general boundary component of small curvature could produce
  billiard tables that admit nonlinearly stable periodic orbits.
\end{abstract}

\maketitle

\tableofcontents

\section{Introduction and main results}

The billiard flow describes the motion of a point particle
along straight lines inside a domain (billiard table) with
specular reflections off the boundary components. The
associated billiard map is the map induced by the billiard flow
going from one reflection off the boundary to the next.
The mathematical tools to investigate dynamical properties of
billiards are best developed for billiards in two-dimensional domains.

The first rigorous proof of hyperbolicity and ergodicity
of some billiards was
carried out by Sinai in \cite{MR274721}. The billiard domains
he considered are now called dispersing or Sinai billiards,
which consist of boundary components with strictly positive curvature
only. One such billiard table is shown in Fig.~\ref{fig_sinai_bunimovich}.
\begin{figure}[!ht]
  \includegraphics[width=0.35\textwidth]{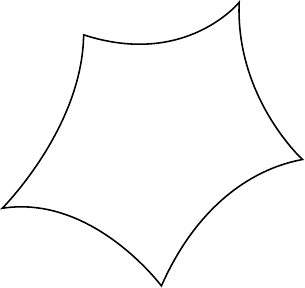}
  \includegraphics[width=0.35\textwidth]{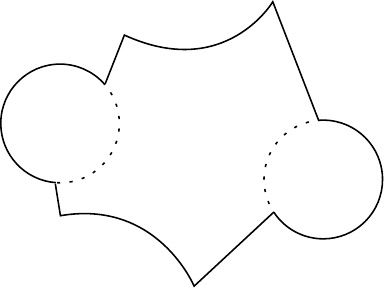}
  \caption{Examples of Sinai billiards (left)
  and Bunimovich billiards (right).}
  \label{fig_sinai_bunimovich}
\end{figure}
Hyperbolicity in dispersing billiards is due to the fact that
dispersing fronts of billiard trajectories remain dispersing and
expand exponentially from collision to collision
in a geometrically intuitive way.

Because of the clear geometric mechanism for hyperbolicity in
dispersing billiards it came as a surprise that just a few years
after Sinai's work Bunimovich proved in
\cites{MR342677,MR0357736,MR530154}
hyperbolicity and ergodicity of billiard tables with focusing boundary
components. In fact, some of the billiard tables he considered were
convex.
An example of such a Bunimovich billiard is shown in
Fig.~\ref{fig_sinai_bunimovich}.
Hyperbolicity in Bunimovich billiards is due to the
so-called defocusing mechanism, which requires that if
fronts of billiard trajectories become focusing due to a reflection
off a focusing boundary component, then
the free path after the last reflection off that boundary component
must be long enough to ensure that the front of billiard trajectories
will become dispersing and expands sufficiently before the next
collision off a boundary component. A convenient way to implement this
mechanism is to allow only arcs of circles for focusing boundary components.
In this case the defocusing mechanism
is guaranteed by an easy to verify geometric condition:
for each circular arc the entire circle is contained inside
the billiard table. This is indicated by the dashed lines in
Fig.~\ref{fig_sinai_bunimovich}.

For a detailed exposition of dispersing billiards and an introduction
to Bunimovich billiards we refer the reader to \cite{MR2229799}.
For hyperbolic billiards with more general types of focusing boundary
components we refer the reader to
\cites{MR848647,MR954676,MR1133266,MR984610,4043756,MR1963965,MR1179172}.

A particular Bunimovich billiard that will be considered in this paper
is a particular version of what is sometimes called flower billiard,
namely a flower billiard with two petals. This billiard table consists
of the union of two overlapping circles. In the special case that these
two circles have the same radius, as is shown in
Fig.~\ref{fig_flower_fold},
\begin{figure}[!ht]
  \includegraphics[width=0.7\textwidth]{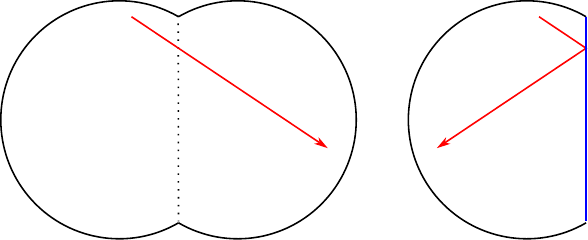}
  \caption{Symmetric flower billiard with two petals and the
  corresponding folded flower table.}
  \label{fig_flower_fold}
\end{figure}
the resulting table has a reflection
symmetry about the dashed line shown in Fig.~\ref{fig_flower_fold}.
Cutting the symmetric two petal flower billiard table along that
axis (and removing the right half)
and inserting a straight line boundary component
we obtain a reduced or folded billiard table, shown in the right side of
Fig.~\ref{fig_flower_fold}. The billiard dynamics in the folded table
is essentially equivalent to the billiard dynamics of the two petal
flower billiard table.
This orbit equivalence is indicated in Fig.~\ref{fig_flower_fold}
by displaying two related billiard trajectory segments.

There are two modifications of the folded flower table
that have been considered in the literature.
One is obtained by replacing the straight line segment
by a circular arc that becomes a focusing boundary component
and thus results in a convex table. These billiards
are sometimes called lemon billiards. The other is obtained
by replacing the straight line by a circular arc that becomes
a dispersing boundary component. These billiards are sometimes called
moon billiards. Both of these are illustrated in
Fig.~\ref{fig_lemon_moon}.
\begin{figure}[!ht]
  \includegraphics[width=0.7\textwidth]{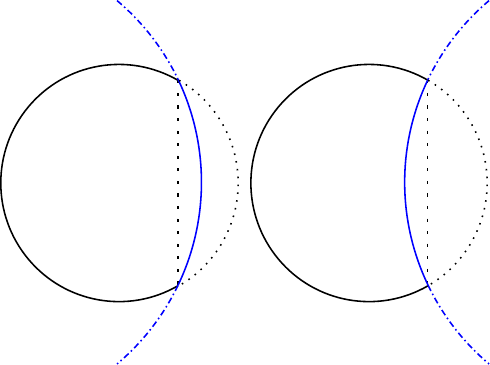}
  \caption{Modifications of the folded flower billiard to a lemon
  billiard and to a moon billiard.}
  \label{fig_lemon_moon}
\end{figure}

It is important to note that neither the lemon billiard nor the moon
billiard is a Bunimovich billiard, and that the defocusing mechanism
fails for both of them.
Therefore, hyperbolicity of either table is not immediately
implied by classic results available in the literature, nor is
there a clear mechanism for hyperbolicity readily available.

If the curvature of the added circular arc is very small we may
think of the resulting lemon or moon billiard as a small perturbation
of the folded flower billiard. And since the folded flower billiard
is hyperbolic, as shown by Bunimovich's work, it might seem
reasonable to expect that perturbations of the folded flower billiard
to either a lemon or moon billiard are also hyperbolic.

Indeed, the hyperbolicity of such lemon billiards was first investigated
in \cite{MR3452271} and later rigorously proven in \cite{MR4183371}.
Numerical evidence for the hyperbolicity of such moon billiards
was established in \cite{MR3456006}. For large curvatures of the
added circular arc the resulting moon billiards can exhibit elliptic
periodic orbits as was shown in \cite{MR3456006}. (See also
\cite{MR4651576} for related work on annular billiards.)

In the present paper we consider other modifications of the folded
flower billiard by replacing the straight line segment not by
arcs of a circle, but by a more general boundary component.
An illustration of such a
$\Gamma$ is shown in Fig.~\ref{fig_setup_basic}.
\begin{figure}[!ht]
  \includegraphics[width=0.5\textwidth]{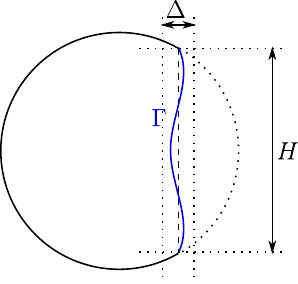}
  \caption{Perturbations of the folded two-petal flower billiard.}
  \label{fig_setup_basic}
\end{figure}
If the curvature of $\Gamma$ is uniformly small
the resulting billiard table is a small geometric perturbation
of the folded flower table, just as the lemon and moon tables were.
Again one might expect that the resulting billiard will be hyperbolic.
However, in the present paper it will be shown that this is not always so.

In order to state a precise version of this we need to introduce some
notation. Let $H>0$ denote the height of the straight line segment in the
folded flower table. Let $\Delta>0$ denote the width of a vertical strip
containing the vertical line segment. Note that for the circular arcs
used to construct the perturbation to a lemon or a moon billiard
the corresponding curvature (in absolute value) multiplied by $H$
is of the same order as the horizontal width it occupies inside
the resulting table divided by $H$.
Therefore, we will consider boundary components $\Gamma$
whose curvature ${\mathcal K}$ (in absolute value)
multiplied by $H$ is of the same order as $\frac{\Delta}{H}$.
And $\Delta$ is taken to be arbitrarily small.
\begin{theorem}
  \label{thm_main}
  There exists a countably infinite set $S$
  such that for each $H \in S$ the following hold.
  \begin{enumerate}
    \item
      \label{item1_thm_main}
      There exists a constant $C>0$ so that
      for all $\Delta>0$ small enough there exists a
      piece-wise smooth $\Gamma$ that is
      entirely contained in the vertical strip of width $\Delta$,
      with
      $|{\mathcal K}|\,H \leq C\,\frac{\Delta}{H}$ on
      $\Gamma$, and such that the resulting perturbed billiard
      table admits a nonlinearly stable periodic orbit $\gamma$ with
      non-vanishing first Birkhoff coefficient.

    \item
      \label{item2_thm_main}
      $\Gamma$ can be taken to be symmetric about the
      horizontal symmetry axis (see Fig.~\ref{fig_setup_basic}), and
      to also satisfy one of the following properties
      \begin{enumerate}
 \item
   $\Gamma$ is everywhere smooth and
   ${\mathcal K} >0$ at all points where $\gamma$
   reflects off $\Gamma$.

 \item
   $\Gamma$ is piece-wise smooth and
   ${\mathcal K} >0$ on $\Gamma$.

 \item
   $\Gamma$ is everywhere smooth, strictly convex with
   ${\mathcal K} < 0$ everywhere on $\Gamma$.
      \end{enumerate}
  \end{enumerate}
\end{theorem}

The existence of $\Gamma$ with all its properties
claimed in Theorem~\ref{thm_main} is stable under small $C^5$ perturbations
of $\Gamma$,
so that we obtain an open set (in the $C^5$ topology) of billiard
tables that contains arbitrarily small (in the sense of curvature)
perturbations of the folded flower table which admit nonlinearly
stable periodic orbits.

The remainder of the paper establishes a proof of Theorem~\ref{thm_main}.
In Section~\ref{sect_construction} we describe the general part of
our construction of $\Gamma$ and the choice of the special
values of $H$ given by $S$ in Theorem~\ref{item1_thm_main} that give
rise to a special periodic orbit $\gamma$.
In Section~\ref{sect_stability} we derive a criterion for the
stability of the periodic orbit $\gamma$.
Both together will provide a proof for Item~(\ref{item1_thm_main}) of
Theorem~\ref{thm_main}. It will become clear that the construction used
does not depend on any symmetry assumptions and is able to produces
asymmetric boundary components $\Gamma$ for which the resulting
billiard table admits stable periodic orbits that are not symmetric either.
For simplicity, however, we shall only consider the more restrictive case
of symmetric $\Gamma$ in detail.

The final Section~\ref{sect_examples} provides examples of boundary
components that satisfy additional properties, and thus completes
the proof of Theorem~\ref{thm_main}. That section is
split into two parts.

In the first part, Section~\ref{sect_examples_pos},
we describe constructions of smooth $\Gamma$ such that
the stable periodic orbit $\gamma$ reflects off $\Gamma$
only in places where the curvature is strictly positive. As will become
clear, our construction will require $\Gamma$ to also
have parts with negative curvature to retain the smoothness of
$\Gamma$. If only a piece-wise smooth $\Gamma$
is required, then we will construct examples of such $\Gamma$
that have positive curvature only.
This latter result is in stark contrast with the hyperbolicity
of perturbations to moon billiards as established numerically in
\cite{MR3456006}.

The second part, found in Section~\ref{sect_examples_convex},
describes the construction of smooth components $\Gamma$ that are
strictly convex.
This result is in stark contrast with the hyperbolicity
of perturbations of the folded flower billiard
to lemon billiards established in \cites{MR3452271,MR4183371}.

\section{Construction of special periodic orbits}

\label{sect_construction}

Fix a circle of radius $r$ and
an integer $N\geq 3$, and consider an $N$-periodic orbit inside the
circle with one of its segments oriented so that it is vertical,
as indicated by the dotted polygon in Fig.~\ref{fig_setup} for
$N=5$.
\begin{figure}[!ht]
  \includegraphics[height=0.3\textheight]{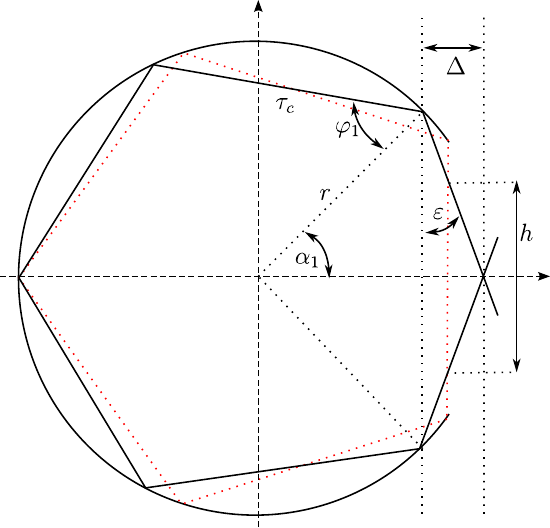}
  \includegraphics[height=0.3\textheight]{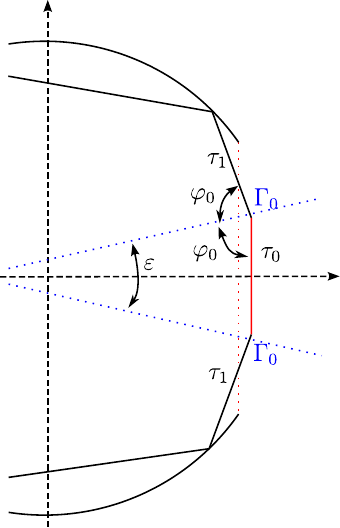}
  \caption{Construction of a billiard orbit segment (left figure)
  that will be used to produce an elliptic periodic orbit of period $N+2$
  (right figure).}
  \label{fig_setup}
\end{figure}
This vertical segment is where the place the cut to make the circle
into a folded flower billiard table. In particular
\begin{equation}
  \label{eqn_H}
  H = 2\,r \sin\frac{\pi}{N}
\end{equation}
is the value of $H$ for that folded flower table.
The collection of all such values of $H$ indexed by $N \geq 3$ provides
the countably infinite set $S$ claimed in Theorem~\ref{thm_main}.

Increasing the angle of reflection while keeping the
symmetry about the horizontal axis results in a billiard orbit segment
having $N$ reflections off the circle as
shown in Fig.~\ref{fig_setup}. In particular, the initially vertical
orbit segment is broken up and is no longer vertical. The resulting angle
with the vertical axis will be denoted by $\varepsilon$. Note that $\varepsilon>0$.
Due to the symmetry the two orbit segments that resulted from
breaking up the initially vertical segment intersect exactly
on the horizontal symmetry axis. The horizontal distance between
that intersection point and the first point of reflection on
the circular arc will be denoted by $\Delta$. Those two orbit segments
also intersect the initially vertical segment. The distance
between these two intersection points will be denoted by $h$.
Furthermore, the polar angle corresponding to the first
reflection off the circular arc is denoted by $\alpha_1$,
and the corresponding angle of reflection is denoted by
$\varphi_1$.

The symmetry about the horizontal axis then requires that
$\varepsilon$, $\varphi_1$ and $\alpha_1$ are related by
$ 2 \alpha_1 + (N-1)(\pi - 2\,\varphi_1) = 2\pi $
and
$ \alpha_1 + \varphi_1 + \frac{\pi}{2} - \varepsilon = \pi $,
i.e.
$\alpha_1 = \frac{\pi}{2} - \varphi_1 + \varepsilon $.
Therefore we can use $\varepsilon$ to express $\alpha_1$ and $\varphi_1$
as follows:
\begin{equation}
  \label{eqn_phi1_alpha1}
  \alpha_1 = \frac{\pi}{N} + \frac{N-1}{N}\,\varepsilon
  \;,\quad
  \frac{\pi}{2} - \varphi_1 = \frac{\pi}{ N } - \frac{1}{N}\,\varepsilon
  \;.
\end{equation}
For the length $\tau_c$ of the free path along the circular arc
we have $\tau_c = 2\,r \cos\varphi_1$, and hence
\begin{equation}
  \label{eqn_tc}
  \tau_c
  =
  2\,r \sin\Big( \frac{\pi}{ N } - \frac{1}{N}\,\varepsilon \Big)
  =
  2\,r \sin\frac{\pi}{ N } + \operatorname{\mathcal{O}}(\varepsilon)
\end{equation}
follows.
The expressions for $\Delta$
\begin{equation}
  \label{eqn_Delta}
  \Delta = r \sin\alpha_1\,\tan\varepsilon
\end{equation}
and $h$
\begin{equation}
  \label{eqn_h}
  h
  =
  2\,r \sin\alpha_1
  -
  2\,r\, \frac{\cos\frac{\pi}{N} - \cos\alpha_1}{ \tan\varepsilon }
  =
  \frac{1}{N}\,\tau_c + \operatorname{\mathcal{O}}(\varepsilon)
  =
  \frac{1}{N}\,2\,r \sin\frac{\pi}{ N } + \operatorname{\mathcal{O}}(\varepsilon)
\end{equation}
are both readily obtained from Fig.~\ref{fig_setup}.

In order to generate a periodic orbit $\gamma$ we add a new boundary component
$\Gamma$ such that the two orbit segments that broke
off the initially vertical segment intersect $\Gamma$.
We may as well make $\Gamma$ symmetric about the horizontal
symmetry axis. Denote the two points of intersection by
$\Gamma_0$ (they are symmetric about the horizontal axis,
so we use the same symbol to denote both) as shown in
Fig.~\ref{fig_setup}. Furthermore, we make the tangent line at
$\Gamma_0$ such that the orbit segments will reflect
off $\Gamma_0$ at an angle $\varphi_0$ that will
produce a vertical outgoing orbit segment. Thus we obtain a
periodic orbit $\gamma$ of period $N+2$, provided that $\Gamma$
is also chosen to not introduce any additional intersection
(would-be reflections) with $\gamma$.

We denote the free path from $\Gamma_0$ to the circular arc
by $\tau_1$, and we use $\tau_0$ to denote the free path
between $\Gamma_0$ and its symmetric copy, which corresponds
to the length of the vertical orbit segment of $\gamma$ connecting
$\Gamma_0$ and its symmetric copy.

It is evident from this construction of the $N+2$ periodic orbit,
see also the illustration shown in Fig.~\ref{fig_setup}, that
$\varphi_1 + \varphi_0 + \alpha_1 - \frac{1}{2}\varepsilon = \pi$.
Thus, using \eqref{eqn_phi1_alpha1} it follows that
\begin{equation}
  \label{eqn_phi0}
  \frac{\pi}{2} - \varphi_0 = \frac{1}{2}\varepsilon
\end{equation}
as indicated in Fig.~\ref{fig_setup}.

There is, however, a relation between $\tau_0$, $\tau_1$
and the value of $\varepsilon$. Clearly, given $\varepsilon$ the value of
$\tau_0$ can be anywhere in the range:
\begin{equation}
  \label{eqn_rangle_tau0}
  0 < \tau_0 < 2\,r \sin\alpha_1
\end{equation}
And for any such choice of $\tau_0$ we then have fixed the location
of $\Gamma_0$. The corresponding normal line must always have
angle $\frac{\varepsilon}{2}$ with the horizontal axis.
The value of $\tau_1$ is then also fixed, and must satisfy
the relation
\begin{equation*}
  r \sin\alpha_1 - \frac{1}{2} \tau_0
  =
  \tau_1 \cos\varepsilon
\end{equation*}
as can be readily seen from Fig.~\ref{fig_setup}. Therefore, we have
\begin{equation}
  \label{eqn_tau1}
  \tau_1
  =
  \frac{ 2\,r \sin\alpha_1 - \tau_0 }{ 2 \cos\varepsilon }
\end{equation}
for $\tau_1$ as a function of $\varepsilon$ and $\tau_0$.

To summarize the above construction, fix any $N \geq 3$
and any (small) value of $\varepsilon>0$. For any choice of
$\tau_0$ in the admissible range \eqref{eqn_rangle_tau0}
there exists an $N+2$ periodic orbit $\gamma$,
that can be chosen to be symmetric
about the horizontal axis, as shown in Fig.~\ref{fig_setup}.
Notice that only the location of $\Gamma_0$ (fixed
by the choice of $\tau_0$), and the corresponding normal line
(whose angle with the horizontal axis is $\frac{\varepsilon}{2}$)
are fixed by the above construction. The local (and global) geometry
of the boundary component $\Gamma$ near $\Gamma_0$
does not affect the periodic orbit $\gamma$, as long as $\Gamma$
does not have other intersections with $\gamma$ (which would prevent
$\gamma$ from existing as described in the above construction).

This establishes the countably infinite set $S$ in
Theorem~\ref{thm_main}, and the existence $\Gamma$ that
gives rise to a periodic orbit $\gamma$. The corresponding stability
claim of Item~(\ref{item1_thm_main}) in Theorem~\ref{thm_main} will be
addressed in the next section.

\section{Stability of the periodic orbit}
\label{sect_stability}

Fix $r$, $N$, $\varepsilon$, $\tau_0$ and consider the periodic orbit $\gamma$
as constructed in Section~\ref{sect_construction}. Recall that the
construction of $\gamma$ described above does not determine the
boundary component $\Gamma$ globally. Only conditions on
the local geometry of any such $\Gamma$ are specified.
Let ${\mathcal K}_0$ denote the curvature of (any one such)
$\Gamma$ at $\Gamma_0$.
While the global description of $\Gamma$ is essential in
order to prove existence of $\gamma$ as described in
Section~\ref{sect_construction}, it is only the local geometry that
matters for the stability analysis of $\gamma$. This will be carried
out in this section, while we postpone the global description of
$\Gamma$ until the next section.

The linearization of the billiard flow,
see \cite{MR2229799} for details, from right before the
first reflection off the circular arc to right after the $N$-th reflection
off the circular arc will be denoted by $L_c$ and is given by
\begin{equation*}
  L_c
  =
  \Big[
  \begin{pmatrix}
    -1 & 0 \\
    \frac{4}{\tau_c} & -1
  \end{pmatrix}
  \begin{pmatrix}
    1 & \tau_c \\
    0 & 1
  \end{pmatrix}
  \Big]^N
  \begin{pmatrix}
    -1 & 0 \\
    \frac{4}{\tau_c} & -1
  \end{pmatrix}
  \;.
\end{equation*}
It is not difficult to establish that
\begin{equation}
  \label{eqn_Lc}
  L_c
  =
  \begin{pmatrix}
    1 - 2\,N & (N-1)\,\tau_c \\
    \frac{4\,N}{\tau_c} & 1 - 2\,N
  \end{pmatrix}
\end{equation}
holds.

Therefore, the linearization of the billiard flow around the entire periodic
orbit $\gamma$, which will be denoted by $M$, is given by
\begin{equation}
  \label{eqn_M}
  M
  =
  L_c
  \begin{pmatrix}
    1 & \tau_1 \\
    0 & 1
  \end{pmatrix}
  \begin{pmatrix}
    1 & 0 \\
    {\mathcal R}_0 & 1
  \end{pmatrix}
  \begin{pmatrix}
    1 & \tau_0 \\
    0 & 1
  \end{pmatrix}
  \begin{pmatrix}
    1 & 0 \\
    {\mathcal R}_0 & 1
  \end{pmatrix}
  \begin{pmatrix}
    1 & \tau_1 \\
    0 & 1
  \end{pmatrix}
  \;,\quad
  {\mathcal R}_0
  =
  \frac{2\,{\mathcal K}_0}{\cos\varphi_0}
  \;,
\end{equation}
where the post-collisional state of the last reflection off the
circular arc was taken as the initial point for linearization of the
billiard flow along $\gamma$.

The stability type (elliptic, parabolic, hyperbolic) of $\gamma$
is determined by the trace of $M$.
In particular, $\gamma$ is elliptic if $-1 < \frac{1}{2} \operatorname{tr} M < 1$.
A straightforward calculation using the explicit formula
\eqref{eqn_M} for $M$ shows the following result:
\begin{lemma}
  \label{lem_trM2}
  The trace $\operatorname{tr} M$ of $M$ is given by
  \begin{equation*}
    \begin{split}
      \frac{1}{2} \operatorname{tr} M
      &=
      1
      +
      \frac{1}{2\,\tau_c}
      [
      2\,N
      -
      (
      (N-1)\,\tau_c - 2\,N\,\tau_1
      )\,{\mathcal R}_0
      ]\,[
      2\,( 2\,\tau_1 + \tau_0 - \tau_c )
      -
      ( \tau_c - 2\,\tau_1 )\,\tau_0\,{\mathcal R}_0
      ]
      \\
      &=
      -
      1
      +
      \frac{1}{2\,\tau_c}
      [
      2
      -
      ( \tau_c - 2\,\tau_1 )\,{\mathcal R}_0
      ]\,[
      2\,( \tau_c + N\,( 2\,\tau_1 + \tau_0 - \tau_c ) )
      -
      ( (N-1)\,\tau_c - 2\,N\,\tau_1 )\,\tau_0\,{\mathcal R}_0
      ]
      \;.
    \end{split}
  \end{equation*}
\end{lemma}

In the following we are interested in
finding conditions on ${\mathcal R}_0$ such that the periodic orbit $\gamma$
is elliptic, i.e. $-1 < \frac{1}{2} \operatorname{tr} M < 1$ holds. For this purpose
we note the following observation:
\begin{lemma}
  \label{lem_t0_2t1_tc}
  The three free paths $\tau_0$, $\tau_1$, $\tau_c$
  satisfy
  \begin{equation*}
    2\,\tau_1 + \tau_0 - \tau_c
    =
    \Big(\frac{1}{ \cos\varepsilon } - 1 \Big)
    \,[ 2\,r \sin\alpha_1 - \tau_0 ]
    + 2\,r\,[ \sin\alpha_1 - \sin(\alpha_1 - \varepsilon) ]
    >
    0
  \end{equation*}
  for all $\varepsilon>0$ (small) and all admissible $\tau_0$.
  In particular,
  $2\,\tau_1 + \tau_0 - \tau_c = \operatorname{\mathcal{O}}(\varepsilon)$
  holds uniformly in $\varepsilon$.
\end{lemma}
\begin{proof}
  By
  \eqref{eqn_tau1}
  and the definition of $\tau_c$ we have
  \begin{equation*}
    2\,\tau_1 + \tau_0 - \tau_c
    =
    \frac{ 2\,r \sin\alpha_1 - \tau_0 }{ \cos\varepsilon }
    + \tau_0
    - 2\,r\cos\varphi_1
    \;.
  \end{equation*}
  Recall also that by \eqref{eqn_phi1_alpha1}
  $ \frac{\pi}{2} - \varphi_1 = \alpha_1 - \varepsilon $, so that
  $\cos\varphi_1 = \sin(\alpha_1 - \varepsilon)$, and therefore
  \begin{equation*}
    2\,\tau_1 + \tau_0 - \tau_c
    =
    \frac{ 2\,r \sin\alpha_1 - \tau_0 }{ \cos\varepsilon }
    + \tau_0
    - 2\,r \sin(\alpha_1 - \varepsilon)
  \end{equation*}
  holds, which can be readily seen to be equivalent to the claimed
  expression for
  $2\,\tau_1 + \tau_0 - \tau_c$.

  Now, by the restriction \eqref{eqn_rangle_tau0} on the range
  for $\tau_0$ it follows that the first term on the right-hand-side
  of the claimed expression for
  $2\,\tau_1 + \tau_0 - \tau_c$
  is positive. The second term is obviously also positive for
  any $\varepsilon>0$ (small).
\end{proof}

It is evident from Lemma~\ref{lem_trM2} that $\operatorname{tr} M$ is a quadratic
function in ${\mathcal R}_0$.
Using the factorization of $\frac{1}{2} \operatorname{tr} M \pm 1$
derived in Lemma~\ref{lem_trM2}
we can readily obtain the solutions ${\mathcal R}_0$ to
$\frac{1}{2} \operatorname{tr} M = \pm 1$.
In light of Lemma~\ref{lem_t0_2t1_tc}
we express these as follows:
\begin{equation}
  \label{eqn_R0_trM_plus2}
  \frac{1}{2} \operatorname{tr} M = 1
  \iff
  {\mathcal R}_0
  =
  R_1
  \;,\quad
  {\mathcal R}_0
  =
  R_2
\end{equation}
\begin{equation}
  \label{eqn_R0_trM_minus2}
  \frac{1}{2}\operatorname{tr} M = -1
  \iff
  {\mathcal R}_0
  =
  R_1 + \frac{2}{\tau_0}
  \;,\quad
  {\mathcal R}_0
  =
  R_2
  -
  \frac{2}{\tau_0}
\end{equation}
where we set
\begin{equation*}
  R_1
  =
  \frac{2}{ \tau_0 - (2\,\tau_1 + \tau_0 - \tau_c) }
  -
  \frac{2}{\tau_0}
  \quad\text{and}\quad
  R_2
  =
  \frac{2
  }{
  \tau_0 - \frac{1}{N}\,\tau_c
  -
  (2\,\tau_1 + \tau_0 - \tau_c)
  }
  \;.
\end{equation*}

Notice that the term in $\frac{1}{2}\operatorname{tr} M$
that is proportional to ${\mathcal R}_0^2$ has coefficient
\begin{equation}
  \label{eqn_coeff_R02}
  \frac{N \,\tau_0 }{2\,\tau_c}
  \,[ \tau_0 - ( 2\,\tau_1 + \tau_0 - \tau_c ) ]
  \,\Big[
  \tau_0
  -
  \frac{1}{N}\,\tau_c
  -
  ( 2\,\tau_1 + \tau_0 - \tau_c )
  \Big]
  =
  \frac{2\,N \,\tau_0 }{\tau_c}
  \,\frac{1}{R_1 + \frac{2}{\tau_0}}
  \,\frac{1}{R_2}
\end{equation}
as is readily seen from Lemma~\ref{lem_trM2}.
Combining
\eqref{eqn_R0_trM_plus2}, \eqref{eqn_R0_trM_minus2}, \eqref{eqn_coeff_R02}
readily shows the following result:

\begin{proposition}
  \label{prop_trM_vs_R0}
  Assume that
  \begin{equation*}
    2\,\tau_1 + \tau_0 - \tau_c
    <
    \tau_0
    \;.
  \end{equation*}
  (Recall that 
  $
  0 <
  2\,\tau_1 + \tau_0 - \tau_c
  = \operatorname{\mathcal{O}}(\varepsilon)
  $
  as was shown in Lemma~\ref{lem_t0_2t1_tc}.)
  Then the following hold:
  \begin{enumerate}
    \item
      \label{item_1_prop_trM_vs_R0}
      The inequality
      \begin{equation*}
 0 < R_1
 =
 \frac{2}{\tau_0}
 \frac{
 (2\,\tau_1 + \tau_0 - \tau_c)
 }{
 \tau_0
 -
 (2\,\tau_1 + \tau_0 - \tau_c)
 }
 =
 \operatorname{\mathcal{O}}(\varepsilon)
      \end{equation*}
      holds.
      Furthermore,
      as ${\mathcal R}_0$ increases from $R_1$ to $R_1 + \frac{2}{\tau_0}$
      the trace of $M$ decreases from $2$ to $-2$.
    \item
      \label{item_2_prop_trM_vs_R0}
      If
      $
      \tau_0
      <
      \frac{1}{N}\,\tau_c
      +
      (2\,\tau_1 + \tau_0 - \tau_c)
      $, i.e. if $R_2 < 0$,
      then ${\mathcal R}_0 \mapsto \operatorname{tr} M$ is a downward pointing parabola,
      and
      as ${\mathcal R}_0$ increases from $R_2 - \frac{2}{\tau_0}$ to $R_2$
      the trace of $M$ increases from $-2$ to $2$.
    \item
      \label{item_3_prop_trM_vs_R0}
      If
      $
      \frac{1}{N}\,\tau_c
      +
      (2\,\tau_1 + \tau_0 - \tau_c)
      <
      \tau_0
      $, i.e. if $R_2 > 0$, then
      then ${\mathcal R}_0 \mapsto \operatorname{tr} M$ is an upward pointing parabola,
      and
      \begin{equation*}
 0
 <
 R_2 - \frac{2}{\tau_0}
 =
 \frac{2}{\tau_0}
 \frac{
 \frac{1}{N}\,\tau_c
 +
 (2\,\tau_1 + \tau_0 - \tau_c)
 }{
 \tau_0 - \frac{1}{N}\,\tau_c
 -
 (2\,\tau_1 + \tau_0 - \tau_c)
 }
 =
 \frac{2}{\tau_0}
 \frac{
 \frac{1}{N}\,\tau_c
 }{
 \tau_0 - \frac{1}{N}\,\tau_c
 }
 +
 \operatorname{\mathcal{O}}(\varepsilon)
 \;,
      \end{equation*}
      and
      as ${\mathcal R}_0$ increases from $R_2 - \frac{2}{\tau_0}$ to $R_2$
      the trace of $M$ increases from $-2$ to $2$.
  \end{enumerate}
\end{proposition}

In order to have a geometric interpretation of the condition
$
\frac{1}{N}\,\tau_c
+
(2\,\tau_1 + \tau_0 - \tau_c)
<
\tau_0
$
that is invoked in Proposition~\ref{prop_trM_vs_R0}
observe that \eqref{eqn_h} combined with Lemma~\ref{lem_t0_2t1_tc}
shows that
\begin{equation}
  \label{eqn_cond_t0}
  \frac{1}{N}\,\tau_c
  +
  (2\,\tau_1 + \tau_0 - \tau_c)
  =
  \frac{1}{N}\,\tau_c + \operatorname{\mathcal{O}}(\varepsilon)
  =
  h + \operatorname{\mathcal{O}}(\varepsilon)
  =
  \frac{1}{N}\,2\,r \sin\frac{\pi}{N}
  +
  \operatorname{\mathcal{O}}(\varepsilon)
  \;.
\end{equation}
Thus we arrive at the following
observation:
\begin{remark}
  \label{rem_trM_vs_R0_cond}
  The geometric meaning of
  $
  \frac{1}{N}\,\tau_c
  +
  (2\,\tau_1 + \tau_0 - \tau_c)
  <
  \tau_0
  $
  is that $\tau_0$ is chosen such that
  it is (within an error of $\operatorname{\mathcal{O}}(\varepsilon)$)
  longer than the $\frac{1}{N}$--th part of the vertical reference
  orbit segment indicated by the dotted vertical red line in
  Fig.~\ref{fig_setup}, i.e. $\frac{1}{N}\,2\,r \sin\frac{\pi}{N}$.
  Equivalently, it means that $\tau_0$
  is chosen longer (within $\operatorname{\mathcal{O}}(\varepsilon)$) than $h$.
\end{remark}

We conclude that Proposition~\ref{prop_trM_vs_R0} provides explicit bounds
on the local geometry of $\Gamma$, indeed on
${\mathcal K}_0$, that guarantee that the periodic orbit
$\gamma$ is elliptic. Furthermore, \eqref{eqn_cond_t0} provides
a clear geometric meaning of the condition distinguishing
Item~(\ref{item_2_prop_trM_vs_R0}) and Item~(\ref{item_3_prop_trM_vs_R0})
in Proposition~\ref{item_2_prop_trM_vs_R0}.

This essentially completes the proof of
Item~(\ref{item1_thm_main}) of Theorem~\ref{item1_thm_main}. What is missing
still is the global construction of $\Gamma$ so that
$\gamma$ is indeed a periodic orbit on the resulting billiard table.
This will be carried out by means of specific examples for
$\Gamma$ in the next
section. This will also address the additional properties that
can be imposed on $\Gamma$ as listed in
Item~(\ref{item2_thm_main}) of Theorem~\ref{item2_thm_main}, and will
also establish the nonlinear stability of $\gamma$
and the non-vanishing of the first Birkhoff coefficient, and thereby
completing the proof of Theorem~\ref{thm_main}.

\section{Examples of boundary components with elliptic periodic orbits}
\label{sect_examples}

Fix $N$, $\varepsilon$, $\tau_0$ and consider the periodic orbit $\gamma$
as constructed in Section~\ref{sect_construction}.
Let ${\mathcal K}_0$ denote the curvature of $\Gamma$
at $\Gamma_0$.
Recall \eqref{eqn_phi0} and note that
\begin{equation*}
  {\mathcal R}_0
  =
  \frac{2\,{\mathcal K}_0}{\cos\varphi_0}
  =
  \frac{2\,{\mathcal K}_0}{\sin\frac{\varepsilon}{2} }
\end{equation*}
holds. In this section we give examples of boundary components
$\Gamma$ that have a uniformly small curvature
and are such that the resulting billiard table admits an
elliptic $N+2$--periodic orbit, as illustrated in Fig.~\ref{fig_setup}.

\subsection{Boundary components with positive curvature}
\label{sect_examples_pos}

By Item~(\ref{item_1_prop_trM_vs_R0}) of
Proposition~\ref{prop_trM_vs_R0}, if the boundary curvature ${\mathcal K}_0$
at $\Gamma_0$ is such that
$R_1 < {\mathcal R}_0 < R_1 + \frac{2}{\tau_0}$, then
the periodic orbit $\gamma$ is elliptic.
In more explicit terms this condition can be written as
\begin{equation*}
  0
  <
  \operatorname{\mathcal{O}}(\varepsilon)
  =
  \frac{2}{\tau_0}
  \frac{
  (2\,\tau_1 + \tau_0 - \tau_c)
  }{
  \tau_0
  -
  (2\,\tau_1 + \tau_0 - \tau_c)
  }
  <
  \frac{2\,{\mathcal K}_0}{\sin\frac{\varepsilon}{2} }
  <
  \frac{2}{ \tau_0 - (2\,\tau_1 + \tau_0 - \tau_c) }
  =
  \frac{2}{ \tau_0 } + \operatorname{\mathcal{O}}(\varepsilon)
  \;.
\end{equation*}
Evidently the value of ${\mathcal K}_0$ can be chosen
in the above range such that $\operatorname{tr} M$ avoids lower
order resonances, i.e.
$\operatorname{tr} M \neq -2, -1, 0, 2$ holds.
Such choice of ${\mathcal K}_0$ clearly can be made
to be uniformly $\operatorname{\mathcal{O}}(\varepsilon)$.

Note, however, that the above requires that ${\mathcal K}_0$
is positive. And since there are two such points, namely
$\Gamma_0$ and its symmetric copy, it is impossible
to have a smooth $\Gamma$ with all positive curvature.
Therefore, imposing global smoothness on $\Gamma$
requires the curvature to become negative in order for $\Gamma$
to connect $\Gamma_0$ and its symmetric copy.

In order for the orbit $\gamma$ to actually exist
it is necessary that the boundary component
$\Gamma$ must not intersect $\gamma$ in places other
than $\Gamma_0$ and its symmetric copy.
This can be guaranteed in many ways.
On such choice of a smooth $\Gamma$ that also stays entirely
withing the vertical strip of width $\Delta$ and is symmetric
about the horizontal axis
is illustrated in Fig.~\ref{fig_example_posK0} below.
\begin{figure}[!ht]
  \includegraphics[height=0.4\textheight]{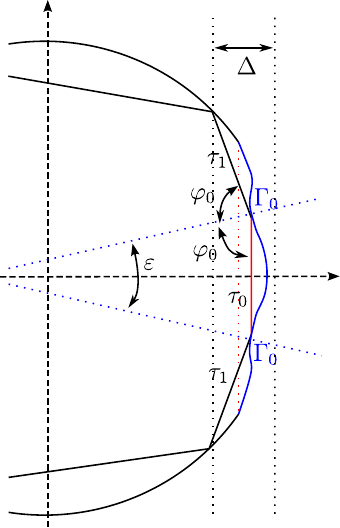}
  \caption{And example of a smooth symmetric
  boundary component $\Gamma$
  that has $\Gamma_0>0$, remains withing the
  vertical strip of width $\Delta$, has curvature uniformly
  bounded by $\operatorname{\mathcal{O}}(\varepsilon)$, and is such that the $N+2$--periodic
  orbit $\gamma$ is nonlinearly stable.}
  \label{fig_example_posK0}
\end{figure}

If $\Gamma$ is required only to be piece-wise smooth,
then $\Gamma$ can have positive curvature everywhere
without intersecting $\gamma$ in places other
than $\Gamma_0$ and its symmetric copy.
One of the many such choices of $\Gamma$ that also stays entirely
withing the vertical strip of width $\Delta$ and is symmetric
about the horizontal axis
is illustrated in Fig.~\ref{fig_example_posK0_piecewise}.
\begin{figure}[!ht]
  \includegraphics[height=0.4\textheight]{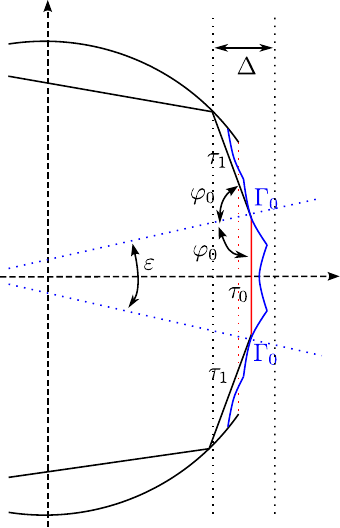}
  \caption{And example of a piece-wise smooth symmetric
  boundary component $\Gamma$ with everywhere positive
  curvature that remains withing the
  vertical strip of width $\Delta$, and has curvature uniformly
  bounded by $\operatorname{\mathcal{O}}(\varepsilon)$, and is such that the $N+2$--periodic
  orbit $\gamma$ is nonlinearly stable.}
  \label{fig_example_posK0_piecewise}
\end{figure}

Regardless of the smoothness condition imposed on $\Gamma$,
the constructions of $\Gamma$ can be carried out
for any one $\varepsilon_0$ small enough. Then we can find a constant
$C>0$ such that
$ |{\mathcal K}|\,H \leq C\,\frac{\Delta}{H} $
holds on $\Gamma$ for $\varepsilon = \varepsilon_0$.
From this we construct $\Gamma$ for all other
$\varepsilon < \varepsilon_0$ by a linear transformation that simply
scales the horizontal direction by the factor
$ \frac{ \tan\frac{\varepsilon}{2} }{ \tan\frac{\varepsilon_0}{2} } $.
This ensures that the tangent at $\Gamma_0$
has angle $\frac{\varepsilon}{2}$ with the vertical direction,
and that the curvature ${\mathcal K}_0$ scales
proportionally to $\varepsilon$ as well as staying within the
admissible range for ellipticity of $\gamma$. Therefore,
$ |{\mathcal K}|\,H \leq (1+C)\,\frac{\Delta}{H} $
holds on $\Gamma$ for all $\varepsilon< \varepsilon_0$,
and completes the entire construction of $\Gamma$
as claimed in Theorem~\ref{thm_main}, except for the nonlinear
stability of $\gamma$.

\begin{remark}
  The construction of tables with elliptic orbits shown in
  Fig.~\ref{fig_example_posK0} and
  Fig.~\ref{fig_example_posK0_piecewise} shows that
  hyperbolicity of moon billiards close to the folded flower billiard table
  (numerical evidence for this is presented in \cite{MR3456006})
  must rely on a mechanism that prevents elliptic periodic like
  the one shown in Fig.~\ref{fig_example_posK0_piecewise} to exist.
  It seems, however, unclear what exactly that mechanism is.
\end{remark}

Using the arguments presented in
\cites{MR2563800,bunimovich2018smoothconvexchaoticbilliard}
the local geometry of boundary component $\Gamma$
can be changed in a neighborhood of $\Gamma_0$
by an arbitrarily small amount in the sense of the $C^5$ topology
such that the resulting $N+2$--periodic orbit $\gamma$ is
nonlinearly stable in the sense that the corresponding
first Birkhoff coefficient can be made non-zero.
This completes the proof of Theorem~\ref{thm_main} for cases of
boundary components with positive curvature at $\Gamma_0$.

\subsection{Strictly convex boundary components}
\label{sect_examples_convex}

Suppose that
$
\tau_0
<
\frac{1}{N}\,\tau_c
+
(2\,\tau_1 + \tau_0 - \tau_c)
$, i.e. suppose $R_2 < 0$.
In more geometric terms this means (recall \eqref{eqn_cond_t0}
and Remark~\ref{rem_trM_vs_R0_cond})
that
$
\tau_0 < h + \operatorname{\mathcal{O}}(\varepsilon)
=
\frac{1}{N}\,2\,r \sin\frac{\pi}{N} + \operatorname{\mathcal{O}}(\varepsilon)
$, and thus corresponds to having chosen $\tau_0$ shorter
than $h$ (see Fig.~\ref{fig_setup}).

By Item~(\ref{item_2_prop_trM_vs_R0}) of Proposition~\ref{prop_trM_vs_R0},
if the boundary curvature ${\mathcal K}_0$
at $\Gamma_0$ is such that
$R_2 - \frac{2}{\tau_0} < {\mathcal R}_0 < R_2$, then
the periodic orbit $\gamma$ is elliptic.
In more explicit terms this condition can be written as
\begin{equation*}
  0
  <
  \frac{2
  }{
  \frac{1}{N}\,\tau_c
  +
  (2\,\tau_1 + \tau_0 - \tau_c)
  -
  \tau_0
  }
  <
  - \frac{2\,{\mathcal K}_0}{\sin\frac{\varepsilon}{2} }
  <
  \frac{2}{\tau_0}
  \frac{
  \frac{1}{N}\,\tau_c
  +
  (2\,\tau_1 + \tau_0 - \tau_c)
  }{
  \frac{1}{N}\,\tau_c
  +
  (2\,\tau_1 + \tau_0 - \tau_c)
  -
  \tau_0
  }
  \;.
\end{equation*}

In order to better understand this condition, note that the point
$\Gamma_0$ and its symmetric copy uniquely determine
a circle that contains both of them and matches the required
normal lines.
This circle has radius ${\rho}_\star$, and thus curvature
${\mathcal K}_\star = -\frac{1}{{\rho}_\star}$, given by
\begin{equation}
  \label{eqn_circle_radius_fit}
  {\rho}_\star 
  =
  \frac{\tau_0}{ 2 \sin\frac{\varepsilon}{2} }
  \;,\quad
  {\mathcal K}_\star = - \frac{ 2 \sin\frac{\varepsilon}{2} }{\tau_0}
  \;.
\end{equation}
Because the vertical line segment of length $\tau_0$ connects
$\Gamma_0$ and its symmetric copy and must correspond
to the radial angle $\varepsilon$ we see that
$ \tau_0 = 2\,{\rho}_\star \sin\frac{\varepsilon}{2} $
must hold.

Then the above condition on the choice of
$- \frac{2\,{\mathcal K}_0}{\sin\frac{\varepsilon}{2} }$
can be rewritten as
\begin{equation}
  \label{eqn_cond_rStar_r0}
  0
  <
  \frac{1}{2}
  \frac{ \tau_0
  }{
  \frac{1}{N}\,\tau_c
  +
  (2\,\tau_1 + \tau_0 - \tau_c)
  -
  \tau_0
  }
  <
  \frac{{\mathcal K}_0}{ {\mathcal K}_\star }
  <
  \frac{1}{2}
  \frac{
  \frac{1}{N}\,\tau_c
  +
  (2\,\tau_1 + \tau_0 - \tau_c)
  }{
  \frac{1}{N}\,\tau_c
  +
  (2\,\tau_1 + \tau_0 - \tau_c)
  -
  \tau_0
  }
  \;.
\end{equation}
Note that if $\tau_0$ is such that
\begin{equation}
  \label{eqn_cond_t0_negK0}
  \frac{1}{2} \Big[
  \frac{1}{N}\,\tau_c
  +
  (2\,\tau_1 + \tau_0 - \tau_c)
  \Big]
  <
  \tau_0
  <
  \frac{2}{3}\Big[
  \frac{1}{N}\,\tau_c
  +
  (2\,\tau_1 + \tau_0 - \tau_c)
  \Big]
  \;,
\end{equation}
which is stronger than our earlier condition made on $\tau_0$
at the beginning of this section,
then
\begin{equation*}
  \frac{1}{2}
  \frac{ \tau_0
  }{
  \frac{1}{N}\,\tau_c
  +
  (2\,\tau_1 + \tau_0 - \tau_c)
  -
  \tau_0
  }
  <
  1
  \quad\text{and}\quad
  1
  <
  \frac{1}{2}
  \frac{
  \frac{1}{N}\,\tau_c
  +
  (2\,\tau_1 + \tau_0 - \tau_c)
  }{
  \frac{1}{N}\,\tau_c
  +
  (2\,\tau_1 + \tau_0 - \tau_c)
  -
  \tau_0
  }
\end{equation*}
hold. Therefore, if \eqref{eqn_cond_t0_negK0} is satisfied, then
\eqref{eqn_cond_rStar_r0} shows that
the choice ${\mathcal K}_0 = {\mathcal K}_\star$
is admissible in order to make the periodic orbit $\gamma$ elliptic.

The geometric meaning of \eqref{eqn_cond_t0_negK0}
is that $\tau_0$ is chosen so that (within $\operatorname{\mathcal{O}}(\varepsilon)$)
$ \frac{1}{2}\,h < \tau_0 < \frac{2}{3}\, h $, i.e.
$
\frac{1}{N}\,r \sin\frac{\pi}{ N }
<
\tau_0
<
\frac{4}{3}\, \frac{1}{N}\,r \sin\frac{\pi}{ N }
$ holds within an error of $\operatorname{\mathcal{O}}(\varepsilon)$.

Now we can construct a specific boundary component
$\Gamma$ that is convex, symmetric about the horizontal
axis, and results in an elliptic $N+2$--periodic orbit $\gamma$.
For example, we can use a circular arc of curvature equal to
${\mathcal K}_\star$ that contains both $\Gamma_0$
and its symmetric copy, and extends a little bit past them. Then
the curvature is reduced symmetrically to obtain a boundary
component $\Gamma$
that does not intersect the orbit $\gamma$ again, and remains strictly
convex by keeping the curvature strictly negative. Clearly, this
can be done such that $\Gamma$ remains within the vertical
strip of width $\Delta$ (see Fig.~\ref{fig_setup}).
An illustration of this construction is shown in
Fig.~\ref{fig_example_negK0}
\begin{figure}[!ht]
  \includegraphics[height=0.4\textheight]{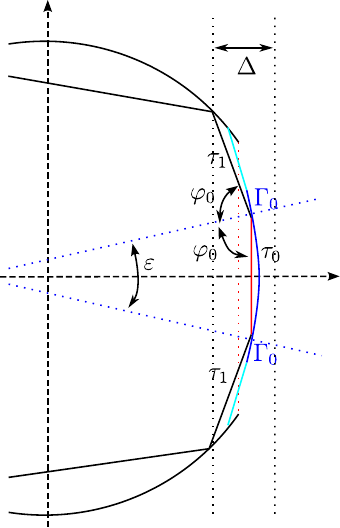}
  \caption{And example of a symmetric strictly convex
  boundary component $\Gamma$
  that remains withing the
  vertical strip of width $\Delta$, has curvature uniformly
  bounded by $\operatorname{\mathcal{O}}(\varepsilon)$, and is such that the $N+2$--periodic
  orbit $\gamma$ is nonlinearly stable.}
  \label{fig_example_negK0}
\end{figure}
Note further that the largest curvature (in absolute value) of
$\Gamma$ is $|{\mathcal K}_0|$, which we
set equal to ${\mathcal K}_\star$. Therefore, using
\eqref{eqn_cond_t0_negK0} to bound $\frac{1}{\tau_0}$,
we see that
\begin{align*}
  \frac{ |{\mathcal K}|\,H }{ \frac{\Delta}{H} }
  &\leq
  \frac{|{\mathcal K}_\star|\,H }{ \frac{\Delta}{H} }
  =
  \frac{ 2 \sin\frac{\varepsilon}{2} }{\tau_0}
  \frac{ 4\,r \sin^2\frac{\pi}{N}
  }{ \sin\alpha_1\,\tan\varepsilon }
  \leq
  \frac{ 2 \sin\frac{\varepsilon}{2} }{
  \frac{1}{2} [
  \frac{1}{N}\,\tau_c
  +
  (2\,\tau_1 + \tau_0 - \tau_c)
  ]
  }
  \frac{ 4\,r \sin^2\frac{\pi}{N}
  }{ \sin\alpha_1\,\tan\varepsilon }
  \\
  &=
  4\,N + \operatorname{\mathcal{O}}(\varepsilon)
\end{align*}
holds on all of $\Gamma$. This shows that there exists
$C>0$ such that for all $\Delta>0$ small enough
$|{\mathcal K}|\,H \leq C\,\frac{\Delta}{H}$ holds on
$\Gamma$. In remains to prove the nonlinear stability of
$\gamma$.

\begin{remark}
  Note that taking $\Gamma$ to be just a circular arc
  of curvature ${\mathcal K}_\star$ would not work, as
  this boundary component would intersect the orbit $\gamma$ again
  before reaching the circular arc, and
  thus making the orbit $\gamma$ not a proper billiard orbit
  on the resulting table. Indeed, it is this mechanism of eliminating
  potentially elliptic orbits that allows for hyperbolicity
  for lemon billiards \cites{MR3452271,MR4183371}
\end{remark}

Evidently the above made choice of ${\mathcal K}_0$
equal to ${\mathcal K}_\star$ is quite arbitrary, and can
be easily adapted to other choices. As described above the resulting
$N+2$--periodic periodic orbit $\gamma$ is such that
$\operatorname{tr} M$ avoids lower order resonances, i.e.
$\operatorname{tr} M \neq -2, -1, 0, 2$ holds.
Furthermore, using the arguments presented in
\cites{MR2563800,bunimovich2018smoothconvexchaoticbilliard}
the boundary component $\Gamma$ as described above
can be changed in a neighborhood of $\Gamma_0$
by an arbitrarily small amount in the sense of the $C^5$ topology
such that the resulting $N+2$--periodic
orbit $\gamma$ is nonlinearly stable with a non-vanishing
first Birkhoff coefficient.
This completes the proof of Theorem~\ref{thm_main} for the case of
strictly convex boundary components with everywhere
strictly negative curvature.

\bibliography{document}
\bibliographystyle{amsplain}

\end{document}